\documentclass[]{amsproc}

\usepackage{amsmath, amssymb,amsfonts}

\usepackage [autostyle, english = american]{csquotes}
\MakeOuterQuote{"}

\usepackage[utf8]{inputenc}
\usepackage[english]{babel}
 
\usepackage{amsthm}

\usepackage{graphicx}

\usepackage{hyperref}

\numberwithin{equation}{section}

\newtheorem{thm}{Theorem}[section]

\newtheorem{prop}[thm]{Proposition}

\newtheorem{lemm}[thm]{Lemma}

\newtheorem{cor}[thm]{Corollary}

\newtheorem{rema}[thm]{Remark}

\newtheorem{defn}[thm]{Definition}

\newcommand{\bq}{\begin{equation}}

\newcommand{\eq}{\end{equation}}

\title{Remarks on functions with bounded Laplacian}
\author{Tarek M. Elgindi}
\address{Department of Mathematics, Princeton University, Princeton, NJ 08544}
\email{tme2@.princeton.edu}
\date{\today}

\begin{document}
\maketitle
\begin{abstract}
$\Delta \psi:=\frac{\partial^2 \psi}{\partial x_1^2}+\frac{\partial^2 \psi}{\partial x_2^2}$ being locally bounded does not imply that $D^2\psi$ is locally bounded. However, we prove that if $\psi$ is invariant under rotation by $\frac{2\pi}{m}$, for some $m\geq 3$, and $\Delta \psi$ is locally bounded, then $$\sup_{x\in B_1(0)}\frac{|\nabla \psi(x)|}{|x|}<\infty.$$ 
This is sharp in that there are examples of functions $\psi$ for which $\Delta \psi$ is locally bounded, which are invariant under rotation by $\pi$ with $|\psi(x)-\psi(0)|\approx |x|^2 |\log|x||$ as $|x|\rightarrow 0$. This bound and its generalizations could be of use in different contexts, particularly for questions about singularity formation in evolution equations. We came upon it while studying certain singular solutions of the incompressible Euler equations in two dimensions (see \cite{E}).  One other application is to prove boundedness of $D^2 \psi$ when $\Delta \psi$ is the characteristic function of a set with self-intersection points (see Section 5). In fact, if $\Delta \psi=\chi_{A}$ and $A$ is the union of sectors emanating from a single point, one can give necessary and sufficient conditions on $A$ for $D^2 \psi$ to be locally bounded (see Section 6). 

\end{abstract}

\tableofcontents

\section{Introduction}

The problem of inverting the Laplacian is one of the fundamental classical problems in mathematical analysis. The basic existence question is: when can we recover a function of $n$ variables, $\psi(x_1,...,x_n),$ from its Laplacian, $\Delta \psi:=\sum_{i=1}^n \frac{\partial^2\psi}{\partial{x_i}^2}$? The related regularity question is: given that $\Delta \psi$ belongs to a certain function space, is it true that $D^2 \psi$ also belongs to that same function space?  There are numerous situations when both the existence and regularity question can be answered affirmatively. For the regularity question, when the function space is $L^p$ for $1<p<\infty$ or $C^\alpha$, $0<\alpha<1$, the answer is generally yes. One notable "borderline" case which is troublesome for the regularity question is when the function space is $L^\infty$.  

More formally, given $f\in L^\infty$ with compact support, consider the problem of finding $\psi$ such that
$$\Delta\psi = f,\,\,\, \text{in} \,\,\, \mathbb{R}^2.$$
Classical results give us that any $\psi$ solving this equation must belong to $W^{2,p}_{loc}$ for every $p<\infty$ (see for example the book of Gilbarg and Trudinger \cite{GilbargTrudinger}). In fact, $D^2\psi\in BMO$, the space of functions of bounded mean-oscillation. However, $D^2\psi$ need not belong to $L^\infty_{loc}$. In fact, $D^2\psi$ may be unbounded near jump discontinuities of $f$, though only logarithmically so. This problem is very much related to the unboundedness of singular integrals on $L^\infty.$ Indeed, using the Newtonian potential, we can write $$\psi(x)= \frac{1}{2\pi} \int \log |x-y| f(y) dy.$$
One can then deduce properties of $D^2 \psi$ from differentiating the integral in the proper way. It can then be seen that $D^2 \psi$ is a matrix of singular integrals applied to $f:$
$$-D^2 \psi=\begin{bmatrix}
R_1^2 f & R_1R_2f \\
       R_1R_2 f    & R_2^2 f  \
    \end{bmatrix}$$
where $R_1$ and $R_2$ are the standard Riesz transforms on $\mathbb{R}^2:$ $$R_i(f):= \text{P.V.}\frac{1}{2\pi}\int_{\mathbb{R}^2}\frac{x_i-y_i}{|x-y|^3} f(y)dy,$$ where the integral is understood in the sense of the principal value. The problem of getting $L^\infty$ bounds on $D^2 \psi$ from $L^\infty$ bounds on $\Delta \psi$ then is equivalent to studying the boundedness of the Riesz transforms $R_1^2, R_2^2,$ and $R_1R_2$ on $L^\infty$. 
\subsection{Borderline Inequalities}
While it would seem that one should just avoid $L^\infty$ when inverting the Laplacian, many problems in the study of non-linear partial differential equations force us to work with $L^\infty$. This tension has begotten many interesting so-called borderline inequalities where authors try to quantify losses incurred by working with $L^\infty.$ A good example of this are the stream of Brezis-Gallouet-Wainger, Moser-Trudinger, and Bourgain-Brezis inequalities (see \cite{BCD}). Though many of those inequalities deal with the lack of critical Sobolev embeddings into $L^\infty$, some of them also deal with the problem of deriving estimates on $D^2 \psi$ from estimates on $\Delta \psi$. For example, Brezis and Gallouet proved that given that $\psi\in C^{2,\alpha}(\overline\Omega)$ and $\psi=0$ on $\partial\Omega$ for $\Omega$ a smooth and bounded domain, 

$$|D^2 \psi|_{L^\infty}\leq C |\Delta \psi|_{L^\infty} \log (\frac{|\psi|_{C^{2,\alpha}}}{|\Delta \psi|_{L^\infty}}+10)$$ for some universal constant $C>0.$ 

The use of this inequality is to be able to assert that while $D^2 \psi$ cannot be bounded point-wise by the maximum of $|\Delta \psi|$ alone, one could do this with a logarithmic loss. Such inequalities play a crucial role in global well-posedness proofs in various contexts in the study of evolution equations.

\subsection{Jump Discontinuities Across Smooth Curves}

In certain situations, $D^2 \psi$ may belong to $L^\infty$ even if $\Delta\psi$ has jump discontinuities. This, however, depends upon the geometry of the jump discontinuities of $\Delta\psi.$ For example, if $\Delta\psi$ is the characteristic function of a set with smooth boundary, it can be shown that $D^2\psi\in L^\infty.$ One does this by observing that if $\Delta\psi$ is characteristic function of a set $A$ with smooth boundary, then $\psi$ is smooth in the direction tangential to the boundary. Looking at the case when $A=\mathbb{R}^2_{+}$, the upper half plane, is quite revealing. In fact, an improvement on the Brezis-Gallouet inequality can actually be proven in this context \cite{Chemin}.\footnote{It must be mentioned that the Brezis-Gallouet inequality can actually be used to control $D^2 \psi$ in the Besov space $B^{0}_{\infty,1}$ which controls any number of Riesz transforms of $D^2 \psi$ in $L^\infty$, whereas anisotropic regularity only allows one to control Riesz transforms of the form $R_1^jR_2^i$ with $i+j\in 2\mathbb{Z}$.} This can be done by direct analysis of the Green's function for the Laplacian as is done in \cite{Chemin},\cite{BertozziConstantin}. This cannot be done, however, when $\Delta \psi$ is the characteristic function of a set which has a corner because there are numerous examples where $\Delta \psi$ is the characteristic function of a set which is smooth away from a single corner and for which $D^2 \psi$ is unbounded near the corner.

\subsection{The Main Result}

Up to now, we have seen two ways to propagate $L^\infty$ bounds on $D^2 \psi$ from bounds on $\Delta \psi$. Both cases require that $\Delta \psi$ enjoy some extra regularity (be it isotropic or anisotropic). There still remains the question: is there a way to control $|D^2 \psi|$  \emph{just} by $|\Delta \psi|_{L^\infty}?$ In this short note, we give one case where this can, effectively, be done.  In fact, we prove that if $\Delta \psi$ is invariant under rotation by $\frac{2\pi}{m}$ for some $m\geq 3$, then, $$\frac{|\nabla \psi(x)|}{|x|} \leq C$$ as $x\rightarrow 0$ \emph{without} any further regularity assumptions on $\psi$.\footnote{Clearly, this can be relaxed to the case where $\Delta \psi$ is invariant under rotation by $\frac{2\pi}{m}$ only in a neighborhood of $0$.}  The following two examples of piece-wise constant functions $f$ illustrate the importance of the symmetry assumption:

\vspace{5mm}
\begin{figure}[h]
{\includegraphics[width=4cm, height=4cm]{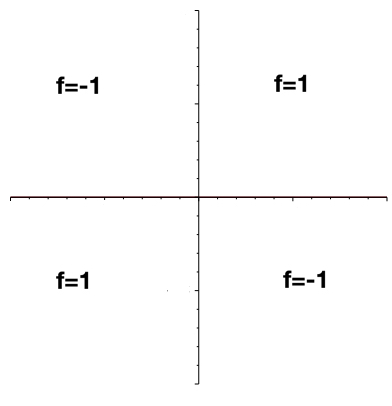}\,\,\,\,\,\,\,\,\,\,\,\,\,\,\,\,\,\,\,\,\,\,\,\,\,\,\,\,\,\,\,\,\,\,\,\,\,\,\,\,\,\,\,\,\,\,\,\,\,\,\,\,\,\,\,\,\,\,\,\,\,\,\,\,\,\,\,\,\,\,\,\,\,\,\,\,\,\,\,\,\,\,\,\,\,\,\,\,\,\,\,\, \includegraphics[width=4cm, height=4cm]{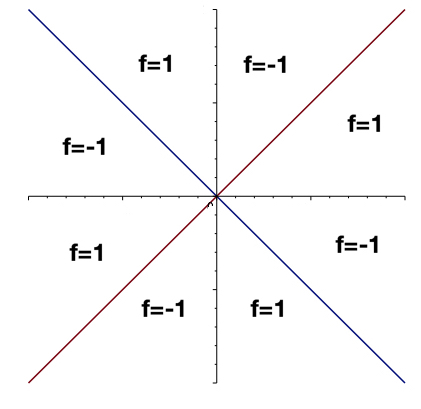}}
\caption{If $\Delta \psi=f$ with $f$ the piecewise constant function on the right, then $D^2 \psi\in L^\infty$. However, if $\Delta \psi=f$ with $f$ the piecewise constant function on the left, then $|\nabla \psi|\approx |x||\log |x||$ near $|x|=0$.}
\end{figure}
\vspace{5mm}

These two examples illustrate the following simple fact: corner-like discontinuities in $\Delta \psi$ may lead to $D^2 \psi$ being unbounded, but only in certain geometric configurations. In other situations, where more symmetries are present near the corner, $D^2\psi$ will remain bounded. 

\subsection{The 2d Incompressible Euler Equations}

We now discuss one particular example of a non-linear evolution equation where information about inverting the Laplacian on $L^\infty$ is very important. Recall the 2d incompressible Euler equations in a smooth bounded domain $\Omega$:
\begin{equation}
\label{2dEuler} \partial_t\omega+u\cdot\nabla\omega=0,
\end{equation}
\begin{equation}
\label{BSL}u=\nabla^\perp\psi,
\end{equation}
where $\psi$ is the unique solution of:
$$\Delta \psi=\omega\,\,\, \text{in}\,\,\, \Omega,$$
$$\psi=0\,\,\, \text{on}\,\,\, \partial\Omega.$$
It is clear that smooth enough solutions of \eqref{2dEuler}-\eqref{BSL} satisfy: $$\frac{d}{dt} |\omega|_{L^p}(t)=0,$$ for all $t\in\mathbb{R}$ and all $1\leq p\leq\infty$. In particular, smooth solutions satisfy an \emph{a-priori} $L^\infty$ estimate. The Yudovich theory (\cite{Yudovich63}) allows us to construct unique global solutions to \eqref{2dEuler}-\eqref{BSL} for given $\omega_0\in L^\infty$. This is due to the fact that an $L^\infty$ estimate on $\omega$ automatically implies a quasi-Lipschitz estimate on $u$ which, in turn, allows us to control the Lagrangian flow-map associated to $u.$ This is all discussed in detail in \cite{Yudovich63} so we do not repeat it here. The $L^\infty$ estimate on $\omega$ also allows us to construct a unique global smooth solution to \eqref{2dEuler}-\eqref{BSL} from smooth initial data. However, due to the fact that the $L^\infty$ estimate on $\omega$ only implies a quasi-Lipschitz bound on $u$ (or using the borderline inequality of Subsection 1.1), one can prove that smooth solutions of \eqref{2dEuler}-\eqref{BSL} satisfy a double exponential type bound:

$$|\nabla\omega(t)|_{L^\infty} \leq |\omega_0|_{L^\infty} \Big(\frac{|\nabla\omega_0|_{L^\infty}}{|\omega_0|_{L^\infty}}\Big)^{\exp(C|\omega_0|_{L^\infty}t)},$$ for some universal constant $C>0.$

The reason that this double exponential growth might be possible is precisely because a bound on $\omega$ in $L^\infty$ does not imply a $W^{2,\infty}$ bound on $\Delta^{-1}\omega$. Hence having a good understanding on when $L^\infty$ bounds on $\omega$ imply a $W^{2,\infty}$ bound on $\Delta^{-1}\omega$ is very important to understand the mechanisms which lead to double exponential growth of $\nabla\omega$.  Improving on this double-exponential estimate in general seems to be a great challenge. In fact, when the 2d Euler equations are studied on smooth bounded domains which are symmetric with respect to some line in $\mathbb{R}^2$, the double exponential bound is known to be optimal in full generality \cite{KS},\cite{Xu}, though the growth is only known to be possible on the boundary. In fact, the constructions in \cite{KS} and \cite{Xu} rely on proving that for certain initial data (satisfying a suitable symmetry condition) $t\rightarrow\infty$, $$\sup_{|x|\leq c} \frac{|\nabla \psi(x)|}{|x|}\geq e^{Ct},$$ which is precluded under our symmetry condition. It is possible that the idea of this paper could be used to preclude double-exponential growth of $\nabla \omega$ away from the boundary. This seems somewhat related to the work of Itoh, Miura, and Yoneda \cite{Yoneda}.

\subsection{Organization of the Paper}

In the second section of this note, we discuss some important examples which will help to build intuition about the sorts of issues one faces when inverting the Laplacian on $L^\infty.$ In the third section we state and prove the main lemma. In the fourth section we discuss various direct extensions. In Sections 5 and 6 we study the problem of solving $$\Delta \psi= \chi_{A}$$ when $A$ is a bounded Lipschitz domain.

\section{Counterexamples}

In this section we discuss several types of examples which illustrate the difficulty of inverting the Laplacian on $L^\infty$.

\subsection{Counterexample 1: Using harmonic polynomials}

Let $P$ be a harmonic polynomial on $\mathbb{R}^2$ which is homogeneous of degree 2. In fact, there are only two possibilities for $P$ (up to scaling): $P(x_1,x_2)=x_1x_2$ or $P(x_1,x_2)=x_1^2-x_2^2$. Note that neither of these are invariant under rotation by $\frac{\pi}{2}$ (they are, in fact, "odd"). Let $\phi$ be a smooth function which is identically 1 in $B_1(0)$ and identically $0$ outside of $B_2(0)$. 
Now define $$\psi_{P}= P(x_1,x_2) \log |x_1^2+x_2^2| \phi(x_1,x_2). $$
Observe that $\psi_P$ is smooth away from $(0,0)$ and that on $B_1(0)$, $\Delta \psi_P=\nabla P \cdot \frac{4x}{|x|^2}.$ Since $P$ is homogenous of degree 2, $\Delta \psi_p\in L^\infty$. On the other hand, clearly $D^2 \psi_P\not\in L^\infty$ and $$|\nabla \psi_P|\approx -|x| \log |x|,$$ near $|x|=0$. 

We note here that this example can be found in the book of Han and Lin \cite{LinHan}. 

\subsection{Counterexample 2: Using Fourier series}

Define $$\psi(x,y)=\sum_{n,m=1}^{\infty} \frac{\sin(nx)\sin(my)}{nm (n^2+m^2)}.$$
Notice that $$\Delta \psi = \sum_{m,n=1}^\infty \frac{\sin(nx)\sin(my)}{nm}= sgn(x)sgn(y)$$ on $[-\pi,\pi]^2$. Hence, $\Delta \psi\in L^\infty$. However, $$\partial_{xy} \psi = \sum_{n,m=1}^\infty\frac{\cos(nx)\cos(my)}{n^2+m^2}.$$ It can be easily checked that $\partial_{xy}\psi$ is like $|\log|x||$ near $|x|=0$. 
Notice that $\psi$ is not invariant under rotation by $\frac{2\pi}{m}$ for any $m\geq 3$.

\subsection{Counterexample 3: The characteristic function of a set with a corner}

Consider, for example, $$\Delta \psi=\chi_{[0,1]^2}.$$ We will show that $$|\nabla \psi-\nabla \psi(0)|\approx |x| |\log|x||.$$
Using the Green's function representation for $\psi$, we get:
$$\partial_{x_1} \psi(x)= \int_{\mathbb{R}^2} \frac{x_1-y_1}{|x-y|^2} \Delta \psi(y) dy= \int_{[0,1]^2 }\frac{x_1-y_1}{|x-y|^2}dy$$
$$=\int_{0}^{1} \int_{0}^1 \frac{x_1-y_1}{(x_1-y_1)^2+(x_2-y_1)^2}  dy=\int_0^1 \frac{1}{2}\log( \frac{(x_2-y_2)^2+(x_1)^2}{(x_2-y_2)^2+(1-x_1)^2})dy_2. $$
Hence, $\partial_{x_1} u(0,0)=\int_0^1\log \frac{y_2^2}{y_2^2+1} dy_2.$ Moreover, $\partial_{x_1} u(0,x_2)=\int_0^1\log\frac{(x_2-y_2)^2}{(x_2-y_2)^2+1}dy_2.$
Therefore, $$\partial_{x_1} u(0,x_2)-\partial_{x_1}u(0,0)=\int_0^1\log \frac{(x_2-y_2)^2}{y_2^2} +\int_0^1\log \frac{y_2^2+1}{(x_2-y_2)^2+1}dy_2$$
$$=\int_0^1\log \frac{(x_2-y_2)^2}{y_2^2}dy_2 + O(x_2).$$
Now, $$\int_0^1 \log\frac{(x_2-y_2)^2}{y_2^2}dy_2= \int_{-x_1}^{1-x_2} \log{y_2^2}dy_2-\int_{0}^{1} \log{y_2^2}dy_2= \int_0^{x_2} \log y_2^2 dy_2-\int_{1-x_2}^{1} \log y_2^2dy_2$$ $$= x_2 \log|x_2^2| +O(x_2). $$
In fact, this could have been done for the case where $\Delta \psi=\chi_{[0,1]^2\cup[0,-1]^2},$ which gives the unboundedness of $D^2 \psi$ in the left figure of Figure 1.

We note here that this example can be found in the paper of Bahouri and Chemin \cite{BC94}.











\subsection{The One-dimensional Case}

We also mention here a short estimate in one dimension which illustrates the use of symmetry conditions.
\begin{lemm}
Let $\phi:\mathbb{R}\rightarrow\mathbb{R}$ be such that $\phi'\in BMO([-1,1]).$ Assume that $\phi$ is even. Then, $$\frac{|\phi(x)-\phi(0)|}{|x|}\leq |\phi'|_{BMO([-1,1])}.$$ 
\end{lemm}
\begin{lemm}
Let $\Phi\in BMO([-1,1])\cap C^{1}([-1,1]-\{0\})$ be an odd function of one variable. Assume that $$|\Phi'(x)|\leq \frac{C}{|x|}$$ for $|x|>0$.  
Then, $\Phi\in L^\infty([-1,1])$ and $$|\Phi|_{L^\infty}\leq 2|\Phi|_{BMO}+ |x \Phi'(x)|_{L^\infty}$$
\end{lemm}
\begin{rema}
Clearly the oddness assumption is essential for both lemmas since $\log|x|\in BMO([-1,1])$ and $\frac{d}{dx} \log|x|=\frac{1}{x}.$ These results should be compared with the higher dimensional case in Lemma 4.4.
\end{rema}
Before giving the proof, we mention a corollary of these lemmas which is a "local" Fefferman-Stein decomposition. 
\begin{cor}
Let $\Phi\in BMO([-1,1])\cap C^{1}([-1,1]-\{0\})$. Assume that $$|\Phi'(x)|\leq \frac{C}{|x|}.$$
Then, there exist $\Phi_1,\Phi_2\in L^\infty([-1,1])$ such that $$\Phi= \Phi_1 + H(\Phi_2),$$ where $H$ is the Hilbert transform. 
\end{cor}
Of course this corollary is implied by the Fefferman-Stein decomposition; however, the proof of this corollary via the preceeding lemmas is much more elementary and proceeds simply by decomposing $\Phi$ into its odd and even parts and then noticing that the Hilbert transform is an isometry and maps odd functions to even functions.

\begin{proof}[Proof of Lemma 2.2]

The first lemma is by definition of the BMO norm and using that $\phi'$ is odd. The second lemma can be seen as follows: Note that $$\frac{1}{\delta}\int_0^\delta|\Phi(t)|dt\leq |\Phi|_{BMO}$$ for all $\delta>0$ since $\Phi$ is odd and using the definition of BMO.  
Let $\epsilon>0$. 
We have that $$\frac{1}{2\epsilon}\int_{0}^{2\epsilon}|\Phi(x)|dx\leq |\Phi|_{BMO},$$ so $$\frac{1}{\epsilon}\int_\epsilon^{2\epsilon}|\Phi(x)|dx\leq 2|\Phi|_{BMO}.$$
But $$|\Phi(x)-\Phi(\epsilon)|\leq \frac{C}{\epsilon}|x-\epsilon|,$$ for all $x>\epsilon$
which implies that $$\frac{1}{\epsilon}\int_\epsilon^{2\epsilon} |\Phi(x)|\geq |\Phi(\epsilon)|-C.$$
In particular, $$|\Phi(\epsilon)|\leq 2|\Phi|_{BMO}+C.$$

\end{proof}

\section{The Lemma and its Proof}

In this section we give the statement of the main lemma and its proof. 

\begin{lemm}

Let $g\in L^\infty(\mathbb{R}^2)$. Assume that for all $x\in\mathbb{R}^2$, $$g(x)=g(x^\perp).$$ 

Let $\psi\in L^\infty_{loc}(\mathbb{R}^2)$ solve $$\Delta \psi= g \,\,\, \text{in} \,\,\,\mathbb{R}^2.$$

Then, $$\sup_{x\in B_1(0)}\frac{|\nabla \psi(x)|}{|x|}<\infty.$$ 
\end{lemm}

\begin{cor}
Assume that $g\in L^\infty(B_1(0))$. Assume that for all $x\in B_1(0)$, $g(x)=g(x^\perp)$. Let $\psi$ be the unique solution of $$\Delta\psi(x)=g(x), \,\,\,|x|<1$$
$$\psi(x)=0, \,\,\, |x|=1.$$
Then, $$\sup_{x\in B_1(0)}\frac{|\nabla\psi(x)|}{|x|}\leq C|g|_{L^\infty},$$ for some universal constant $C>0$. 
\end{cor}

\begin{rema}
We have proven the lemma here in full detail for the case when $g$ enjoys four-fold symmetry. The general case of $n-$ fold symmetry for $n\geq 3$ is handled in Section 4.
Note that it need not be assumed that $g$ satisfies the symmetry condition in all of $\mathbb{R}^2$, just in a neighborhood of the origin. This says that if $g$ is symmetric in a neighborhood of any of its jump discontinuities, the corresponding $\psi$ will satisfy the bounds in the lemma. 
\end{rema}
\begin{rema}
The proof we give is based on the Green's function for the Laplacian. It is possible that other proofs could be given that might give some insight into how this result could be extended to more general linear elliptic equations. Extending this result to, say, divergence-form elliptic equations with bounded measurable coefficients (under some symmetry assumption) might be possible.   
\end{rema}
\begin{rema}
This result can be seen as dual to the "Key Lemma" of  Kiselev and Sverak in \cite{KS}.
\end{rema}


\begin{proof}
Standard considerations about harmonic functions allow us to assume that $g$ actually has compact support. Using the Green's function for the Dirichlet problem on $\mathbb{R}^2$ we see:
$$\nabla \psi(x)=\int_{\mathbb{R}^2} \frac{x-y}{|x-y|^2}g(y)dy.$$
By using the symmetries of $g$, we get:
$$\int_{\mathbb{R}^2}\Big( \frac{(x-y^\perp)}{|x-y^\perp|^2}+ \frac{(x-y)}{|x-y|^2}+ \frac{(x+y^\perp)}{|x+y^\perp|^2}+ \frac{(x+y)}{|x+y|^2}\Big)g(y)dy.$$
Now we simplify the kernel:
$$\frac{(x-y^\perp)}{|x-y^\perp|^2}+ \frac{(x-y)}{|x-y|^2}+ \frac{(x+y^\perp)}{|x+y^\perp|^2}+ \frac{(x+y)}{|x+y|^2}$$ $$= \frac{|x+y|^2(x-y)+|x-y|^2(x+y)}{|x-y|^2|x+y|^2}+ \frac{|x+y^\perp|^2(x-y^\perp)+|x-y^\perp|^2(x+y^\perp)}{|x-y^\perp|^2|x+y^\perp|^2}$$
$$= \frac{x ( 2|x|^2+2|y|^2)-4y(x\cdot y)}{|x-y|^2|x+y|^2}+ \frac{x ( 2|x|^2+2|y|^2)-4y^\perp(x\cdot y^\perp)}{|x-y^\perp|^2|x+y^\perp|^2}$$
$$= 2x(|x|^2+|y|^2)(\frac{1}{|x-y|^2|x+y|^2}+ \frac{1}{|x-y^\perp|^2|x+y^\perp|^2})-\frac{4y (x\cdot y)}{|x-y|^2|x+y|^2}-\frac{4y^\perp (x\cdot y^\perp)}{|x-y^\perp|^2|x+y^\perp|^2}$$
Now, $$|x-y|^2|x+y|^2+|x-y^\perp|^2|x+y^\perp|^2=2(|x|^2+|y|^2)^2-4(x\cdot y)^2-4(x\cdot y^\perp)^2=2(|x|^2 +|y|^2)^2-4|x|^2|y|^2$$ $$=2(|x|^4+|y|^4).$$
Moreover, $$4y^\perp(x\cdot y^\perp)|x-y|^2|x+y|^2+4y(x\cdot y)|x-y^\perp|^2|x+y^\perp|^2$$ $$=y^\perp(x\cdot y^\perp)((|x|^2+|y|^2)^2-2(x\cdot y)^2)+ y(x\cdot y)((|x|^2+|y|^2)^2-2(x\cdot y^\perp)^2).$$
So combining the above terms we get the following kernel: $$\frac{4x(|x|^2+|y|^2)(|x|^4+|y|^4)-4y^\perp(x\cdot y^\perp)((|x|^2+|y|^2)^2-2(x\cdot y)^2)-4y(x\cdot y)((|x|^2+|y|^2)^2-2(x\cdot y^\perp)^2)}{|x-y|^2|x+y|^2|x-y^\perp|^2|x+y^\perp|^2}.$$

Let's focus on the lowest order terms in $x$. We see that the first order terms are:
$$4x|y|^6-4y^\perp(x\cdot y^\perp)|y|^4-4y(x\cdot y)|y|^4=0.$$ 
The third order terms are:
$$4x|x|^2|y|^4-4y^\perp(x\cdot y^\perp)(2|x|^2|y|^2-2(x\cdot y)^2)-4y(x\cdot y)(2|x|^2|y|^2-2(x\cdot y^\perp)^2)$$
$$=-4x|x|^2|y|^4+8y^\perp(x\cdot y^\perp)(x\cdot y)^2+8y(x\cdot y) (x\cdot y^\perp)^2$$
The fifth order terms are:
$$ 4x|x|^4|y|^2-4y^\perp(x\cdot y^\perp)|x|^4-4y(x\cdot y)|x|^4=0.$$
The seventh order term is:
$$4x|x|^6.$$
Hence we get that the kernel simplifies to:
$$\frac{4x|x|^6-4x|x|^2|y|^4+8y^\perp(x\cdot y^\perp)(x\cdot y)^2+8y(x\cdot y)(x\cdot y^\perp)^2}{|x-y|^2|x+y|^2|x-y^\perp|^2|x+y^\perp|^2}$$
Next, our goal is to show:
\begin{equation}\label{GeneralInequality} \int_{B_{10}(0)}|\frac{4x|x|^6-4x|x|^2|y|^4+8y^\perp(x\cdot y^\perp)(x\cdot y)^2+8y(x\cdot y)(x\cdot y^\perp)^2}{|x-y|^2|x+y|^2|x-y^\perp|^2|x+y^\perp|^2}|dy \leq C |x|\end{equation} for some universal constant $C$. This will imply the lemma.
This can actually be done easily by scaling. Indeed, changing variables in the integral $y\rightarrow z |x|,$ we see that \eqref{GeneralInequality} becomes:
$$|x|\int_{B_\frac{10}{|x|}(0)} \Big|\frac{4\frac{x}{|x|}-4\frac{x}{|x|}|z|^4+8z^\perp(\frac{x}{|x|}\cdot z^\perp)(\frac{x}{|x|}\cdot z)^2+8z(\frac{x}{|x|}\cdot z)(\frac{x}{|x|}\cdot z^\perp)^2}{|\frac{x}{|x|}-z|^2|\frac{x}{|x|}+z|^2|\frac{x}{|x|}-z^\perp|^2|\frac{x}{|x|}+z^\perp|^2}\Big|dz.$$

For fixed $x$, this integral is bounded by a universal constant $C$. This is because we may split the integral into an integral in the region where $|z|<2$, where there are singularities in the denominator, and an integral in the region $|z|>2$ where there are no singularities in the denominator. The $|z|>2$ piece is clearly bounded since the integrand is like $|z|^{-4}$ for $|z|>2$. 
The $|z|<2$ case is bounded because the singularities are far from each other so we may restrict to only one of them at a time, say $z\approx \frac{x}{|x|}.$ In this case the numerator is actually 0 hence the singularity becomes like $\frac{1}{|z-\frac{x}{|x|}|}$ which is integrable in two dimensions with a bound independent of $x$. The cases when $z\approx -\frac{x}{|x|}$, $z\approx \frac{x^\perp}{|x|}$, and $z\approx -\frac{x^\perp}{|x|}$ are similar.  

To clarify this point further, we can take the simple case of $x= |x|(1,0).$
In this case, the integral in question becomes:
$$\int_{B_\frac{10}{|x|}(0)} \Big|\frac{(1,0)(4-4|z|^4)+8z^\perp((1,0)\cdot z^\perp)((1,0)\cdot z)^2+8z((1,0)\cdot z)((1,0)\cdot z^\perp)^2}{|(1,0)-z|^2|(1,0)+z|^2|(1,0)-z^\perp|^2|(1,0)+z^\perp|^2}\Big|dz$$ which is controlled by a constant multiple of:
$$\int_{\mathbb{R}^2}\frac{|1-|z|^4|}{|(1,0)-z|^2|(1,0)+z|^2|(1,0)-z^\perp|^2|(1,0)+z^\perp|^2}dz$$ $$+\int_{\mathbb{R}^2}|z|\frac{|(1,0)\cdot z^\perp||(1,0)\cdot z|^2}{|(1,0)-z|^2|(1,0)+z|^2|(1,0)-z^\perp|^2|(1,0)+z^\perp|^2}dz $$ $$+\int_{\mathbb{R}^2}|z|\frac{|(1,0)\cdot z| |(1,0)\cdot z^\perp|^2}{|(1,0)-z|^2|(1,0)+z|^2|(1,0)-z^\perp|^2|(1,0)+z^\perp|^2}dz$$
It is clear that all these quantities are finite as explained above. 

\end{proof}

\section{Extensions}

In this section we discuss a few extensions of Lemma 3.1.
\begin{lemm}
Let $g\in L^1\cap L^\infty(\mathbb{R}^2)$. Assume that $g$ is $m-$fold symmetric, for some $m\in \mathbb{N}$, $m\geq 3$. 
Let $\psi\in L^\infty$ solve $$\Delta \psi =g.$$
Then, $$\sup_{x\in B_1(0)}\frac{|\nabla \psi(x)|}{|x|}+ \sup_{x\in B_1(0)} \frac{|\psi(x)-\psi(0)|}{|x|^2} <\infty.$$ 
\end{lemm}

\begin{rema}
Note that the lemma is false in the case $m=2$ via the counterexamples in Section 2. 
\end{rema}

\begin{proof}

As before, $$\nabla \psi= \int_{\mathbb{R}^2} \frac{x-y}{|x-y|^2} g(y)dy.$$
Since $g$ is $m-$symmetric, 
$$\nabla \psi = \frac{1}{m} \int_{\mathbb{R}^2} \sum_{i=1}^m\frac{x-\mathcal{O}^iy}{|x-\mathcal{O}^iy|^2}g(y)dy.$$

\noindent\emph{Claim:}

$$\frac{1}{m} \int_{B_{10}(0)}\Big|\sum_{i=1}^m \frac{x-\mathcal{O}^iy}{|x-\mathcal{O}^iy|^2} \Big|  dy \leq C |x|.$$ for some universal constant $C$. 

\vspace{5mm}

\noindent\emph{Proof of the Claim:}

Let's first see why this fails when $m=2$. This is because we get:
$$\frac{1}{2}\int_{B_{10}(0)}|\frac{(x+y)|x-y|^2+(x-y)|x+y|^2}{|x-y|^2|x+y|^2}|dy=\int_{B_{10}(0)}|\frac{2x(|x|^2+|y|^2)-4y(x\cdot y)}{|x-y|^2|x+y|^2} |dy.$$
Now we scale in $|x|$ and we see:
$$|x| \int_{B_{\frac{1}{|x|}}}|\frac{2\frac{x}{|x|}(1+|z|^2)-4z(\frac{x}{|x|}\cdot z)}{|\frac{x}{|x|}-z|^2 |\frac{x}{|x|}+z|^2}dz.$$
The problem is that for $z$ very large (up to $\frac{1}{|x|}$) there is a large contribution from the integral which is of size $\log\frac{1}{|x|}$ since the integrand is only like $|z|^{-2}$ which is not integrable in two dimensions for $z$ large. Hence, to show the boundedness for higher symmetries we will need to focus on cancellations for the highest order terms in $y$.

For $m\geq 3$ we see: 

$$\sum_{i=1}^m \frac{x-\mathcal{O}^iy}{|x-\mathcal{O}^iy|^2}=\frac{\sum_{i=1}^m(x-\mathcal{O}^iy)\Pi_{j\not=i}|x-\mathcal{O}^jy|^2}{\Pi_{i=1}^m |x-\mathcal{O}^iy|^2},$$
where we denote by $\mathcal{O}$ the rotation matrix which rotates vectors in $\mathbb{R}^2$ counterclockwise by $\frac{2\pi}{m}$.
Remember that we are only concerned with the highest order terms in $y$ (which are the lowest order terms in $x$). We get that the highest order terms in $y$ of $\sum_{i=1}^m (x-\mathcal{O}^iy) \Pi_{j\not=i}|x-\mathcal{O}^jy|^2$ are:
$$m x |y|^{2m-2} + \sum_{i=1}^m \mathcal{O}^i y|y|^{2m-4}\sum_{j\not=i}2(x,\mathcal{O}^jy)= |y|^{2m-4} \Big( mx|y|^2+ 2\sum_{i=1}^{m} \sum_{j\not=i}\mathcal{O}^i y (\mathcal{O}^jy,x)  \Big)$$
Note that $\sum_{i\not=j} O^jy=-O^i y.$ Hence, $$\sum_{i=1}^{m}\sum_{i\not=i} O^i y(O^jy,x)=-\sum_{i=1}^m\mathcal{O}^iy(\mathcal{O}^iy,x).$$
Now we use a simple lemma.

\begin{lemm}
For $m\geq 3,$ let $\mathcal{O}$ be the linear transformation on $\mathbb{R}^2$ which rotates vectors counterclockwise by $\frac{2\pi}{m}.$ Then, for any $x,y\in \mathbb{R}^2$, 
$$\sum_{i=0}^{m-1} \mathcal{O}^i y(\mathcal{O}^iy,x)=\frac{m}{2} x |y|^2.$$
\end{lemm}

Assuming this Lemma is true, we then see that the highest order terms in $y$ disappear. One can then follow the proof of Lemma 3.1 to conclude. We leave the details to the interested reader. 
\end{proof}

\begin{proof} [Proof Of Lemma 4.3]

First, by scaling it suffices to assume that $|y|=1$. Furthermore, since rotation matrices on $\mathbb{R}^2$ commute, it suffices to assume that $y=(1,0)$. 
The lemma then comes down to proving the following three identities:
$$\sum_{i=0}^{m-1} (\mathcal{O}^i (1,0)\cdot(1,0))^2 = \frac{m}{2},$$ 
$$\sum_{i=0}^{m-1} (\mathcal{O}^i (0,1)\cdot(0,1))^2 = \frac{m}{2},$$ 
and 
$$\sum_{i=0}^{m-1}(\mathcal{O}^i (1,0)\cdot (0,1))(\mathcal{O}^i (1,0)\cdot (1,0))=0.$$
These are equivalent to the three identities:
$$\sum_{i=0}^{m-1} \cos^2(\frac{2\pi i}{m})=\frac{m}{2},$$
$$\sum_{i=0}^{m-1} \sin^2(\frac{2\pi i}{m})=\frac{m}{2},$$ and 
$$\sum_{i=0}^{m-1} \sin(\frac{4\pi i}{m})=0,$$ respectively. Here is where the difference between $m=2$ and $m\geq 3$ becomes manifest. These identities can be checked using formulas for geometric sums.  

\end{proof}

\subsection{$L^\infty$ estimates for $D^2\psi$ when $g$ is smooth away from isolated points}

A natural question which could be asked is: \emph{If $\psi$ and $g$ are as in Lemma 4.1 and $g\in C^{\infty}(\mathbb{R}^2-\{0\})$ is $D^2 \psi$ necessarily bounded?} While it might seem that this is true, counterexamples can be constructed rather simply. However, if we give a natural constraint on the growth of $\nabla g$ near $0$, then $D^2 \psi$ is necessarily bounded.

\begin{lemm}
Let $g\in C^{\infty}_c(\mathbb{R}^2-\{0\})\cap L^\infty(\mathbb{R}^2).$ Assume that $g$ is $m$-fold symmetric near $0$ for some $m\geq 3$ and $$|x||\nabla g(x)|\in L^\infty_{loc}.$$  Let $\psi$ be an $L^\infty_{loc.}$ solution of:
$$\Delta \psi=g$$
on $\mathbb{R}^2.$
Then, $$D^2 \psi\in L^\infty_{loc}(\mathbb{R}^2).$$
\end{lemm}

\begin{rema}
The assumptions of the lemma are sharp in the sense that there are, for each $\alpha>0$, counterexamples with $|\nabla g(x)|\approx |x|^{-1-\alpha}$ near 0.  
\end{rema}

\begin{proof}
By rotation invariance it suffices to prove that $$\int_{\mathbb{R}^2} \frac{(x_1-y_1)(x_2-y_2)}{|x-y|^4} (g(y)-g(x))dy$$ is bounded independent of $x$. By assumption, the following must hold if $|x-y|\leq\frac{|x|}{2}:$ $$|g(y)-g(x)|\leq C \frac{|x-y|}{|x|} .$$
In particular, if $|x-y|$ is much smaller than $|x|$, we can replace one of the powers of $|x-y|$ in the singularity by $|x|$ and be essentially in the same scenario as Lemma 4.1. In the case when $|x-y|\geq \frac{|x|}{2},$ the symmetry can be used to gain enough powers of $|x|$ in the numerator to, again, remove the singularity. 
\end{proof}

\subsection{A counterexample}

We will now construct a counterexample to Lemma 4.4 when the growth assumption on $\nabla g$ near 0 is removed. 

\begin{prop}
There exists $\psi$ which is 4-fold symmetric and smooth away from 0 for which $\Delta \psi\in L^\infty$ but $D^2 \psi\not\in L^\infty_{loc}$.
\end{prop}

\begin{proof}

Consider the function $\tilde f^\epsilon(x_1,x_2)=\epsilon(x_1-\epsilon)x_2 \log ((x_1-\epsilon)^2+x_2^2 +e^{-1/\epsilon^2}) \phi^\epsilon(x_1,x_2)$ where $\phi$ is a smooth cut-off function satisfying $\nabla\phi^{\epsilon}\approx \frac{1}{\epsilon}$ and for which, $\phi^\epsilon\equiv 1$ in $B_{\frac{\epsilon}{10}}(\epsilon,0)$ and 0 outside of $B_{\frac{\epsilon}{2}} (\epsilon,0).$ 
Define $$f^\epsilon(x)=\tilde f^\epsilon(x)+\tilde f^\epsilon (x^\perp)+ \tilde f^\epsilon(-x)+\tilde f^\epsilon(-x^\perp).$$
It is clear that $f$ is 4-fold symmetric.
Moreover, we have $$|\Delta f^\epsilon|\leq \epsilon |\log |\epsilon||. $$
However, $$D^2 f^\epsilon(\epsilon,0)\approx \frac{1}{\epsilon}.$$
Now define $$u= \sum_{N=1}^\infty f^{\frac{1}{100^N}}.$$
Then, $$|\Delta \psi|_{L^\infty}\leq C,$$ $\psi$ is 4-fold symmetric, and $\psi$ is smooth away from $0$ but $D^2 \psi\not\in L^\infty$.  
\end{proof}

\section{An Application to Characteristic Functions of Domains with Corners}

In this section we prove that if we solve $\Delta \psi=\chi_{A}$ with $A$ a bounded Lipschitz domain whose boundary is $C^{1,\alpha}$ away from finitely many corners and which is symmetric around the corners, then $D^2\psi$ is uniformly bounded. Before proving this, we must define what we mean by corners. One way to do this is to say that locally (near the corner) $A$ just consists of two straight lines which form a corner. Taking this definition of a corner actually makes the result too restrictive since it will boil down to a situation like the one in Figure 1 on the right. We will call a corner formed by the intersection of finitely many lines of different slope a polygonal corner.  We will say that $A$ has a corner at a point if $\partial A$ consists of $C^{1,\alpha}$ curves which intersect at that point with different slopes. 

\begin{lemm}
Let $A\subset \mathbb{R}^2$ be a Lipschitz domain which is $C^{1,\alpha}$ away from a single corner at $(0,0)$. Assume that $A$ is invariant under rotation by $\frac{2\pi}{m}$ for some $m\geq 3$.  Let $u\in L^\infty$ solve
$$\Delta \psi=\chi_{A}.$$
Then, $$D^2\psi\in L^\infty$$ for some constant $C$ depending on $A$.   
\end{lemm}

\begin{proof}
We only give a basic sketch of the proof since the issue of domains with corners will be taken up in much greater detail in a later work.
Let $A$ be a set which has a corner at $(0,0)$ and for which $\partial A$ is $C^{1,\alpha}$ away from $(0,0)$.
It is clear (using \cite{BertozziConstantin}) that $D^2\psi$ is locally bounded away from $(0,0)$ so we only need to show that $D^2\psi$ is bounded in a neighborhood of $(0,0)$. 
The proof will proceed by "cutting out" the corners. Since the problem is rotational invariant, we can assume that the line $y=0$ bisects the corner. Call the angle at the corner $2\theta$.

Since the boundary of $A$ near the corner just consists of  $C^{1,\alpha}$ curves which intersect at the corner, $A$ can be written as a union of two pieces: one which is a polygonal corner in a small neighborhood of $(0,0)$ and one an error term. The area of the error region intersected with any small ball of radius $r$ centered at $(0,0)$ is on the order of $r^{2+\alpha}$. The part of $D^2 \psi$ coming from the error term can easily be shown to be bounded. This then reduces the problem to proving boundedness for the case of a symmetric polygonal corner. This is a computation which is in the appendix. 
More specifically, it suffices to consider the boundedness of the following integral independent of $x$:
$$\int_0^1 \int_0^{\phi(y_1)} \sum_{i=0}^{m-1}\frac{(x_1-(\mathcal{O}^iy)_1)(x_2-(\mathcal{O}^iy)_2)}{|x-\mathcal{O}^iy|^4} dy_2 dy_1$$ with $\phi$ some $C^{1,\alpha}$ function with $\phi(0)=0$, $\phi'(0)=c$ for some constant $c$ which relates to the angle $\theta$.
Hence we can break this integral into:
$$\int_0^1 \int_0^{y_1} \sum_{i=0}^{m-1}\frac{(x_1-(\mathcal{O}^iy)_1)(x_2-(\mathcal{O}^iy)_2)}{|x-\mathcal{O}^iy|^4} dy_2 dy_1+\int_0^1 \int_{y_1}^{\phi(y_1)} \sum_{i=0}^{m-1}\frac{(x_1-(\mathcal{O}^iy)_1)(x_2-(\mathcal{O}^iy)_2)}{|x-\mathcal{O}^iy|^4} dy_2 dy_1, $$
the first integral representing the polygonal corner and the second representing the error since $\phi(y_1)-y_1= O(y_1^{1+\alpha}).$
The boundedness of the first term is a computation which is similar to the one in Section 4 and we leave its details for the appendix. The error term (which is also done in \cite{EJ2}) can be controlled using the ideas in Proposition 1 of \cite{BertozziConstantin}, for example.

\end{proof}

\subsection{Example of Boundedness: The characteristic function of a 3-petal flower}

We mention one example where $W^{2,\infty}$ bounds on $\psi$ are possible using Lemma 5.1. Indeed, consider  $$\Delta \psi= \chi_{A}$$ where $A$ is a three-petal flower as given in the following figure.
\begin{figure}[h]
\includegraphics[width=3cm, height=2cm]{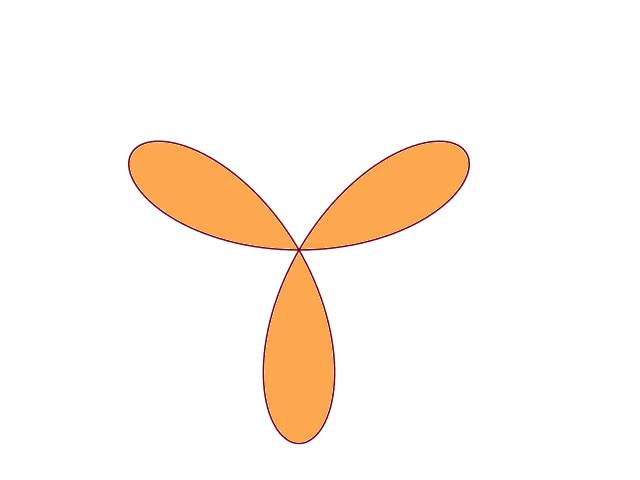}
\end{figure}

\noindent This is the area inside the graph of $$r= \sin(3\theta)$$
in polar coordinates. As a consequence of Lemma 5.1, $$|u|_{W^{2,\infty}}<\infty.$$
The following is an immediate corollary to the proof of Lemma 5.1:

\begin{cor}
Let $A\subset \mathbb{R}^2$ be a bounded Lipschitz domain with a $C^{1,\alpha}$ boundary away from finitely many points $a_1,...,a_N.$ Assume further that for each $a_i$, there exists a neighborhood $B_i$ of $a_i$ inside which $A$ is $m$-fold symmetric with $m\geq 3$. Let $u\in L^\infty$ solve:

$$\Delta \psi=\chi_{A}.$$

Then, $$|u|_{W^{2,\infty}}\leq C$$ for some constant $C$ depending on $A$.   
\end{cor}

\begin{rema}
Note that $A$ could have different symmetries at different corners. For example, $A$ could have a 3-fold symmetry close to one corner and a 4-fold symmetry close to another corner. 
\end{rema}

Further results such as necessary and sufficient conditions for characteristic functions of sets, $A$, which are unions of sectors emanating from the origin will be given in the next section.

\section{Classification of Sets $A$ with $\psi_{A}\in W^{2,\infty}$}
Let $\psi_A$ denote the solution of $$\Delta \psi_{A} =\chi_{A}$$ which is unique up to a harmonic function.
A natural question one could ask is: 
\begin{center}\emph{Is it possible to classify the sets $A$ for which $\psi_A\in W^{2,\infty}$?}\end{center}
This section will be devoted to exploring this question. 
\begin{defn}
Let $0<\alpha<\beta<2\pi$ be such that $\alpha+\beta\leq 2\pi.$ Define the set $S^{\alpha}(\beta)$ as follows:
$$S^\alpha(\beta):=\{(r,\theta)\in\mathbb{R}^2: \,\,\, 0\leq r\leq 1,\,\,\,\ \beta-\alpha<\theta<\beta+\alpha\}.$$
\end{defn}
Pictorially, this is the sector of the unit circle on $\mathbb{R}^2$ which is centered at the line $\theta=\beta$ with angle $2\alpha$. 

Direct calculations show that if $A=S^{\alpha}(\beta)$ for given $\alpha$ and $\beta$, $u_{A}\not\in W^{2,\infty}(B_\frac{1}{2}(0)).$ However, the symmetry conditions of the previous sections tell us that if $A$ is a union of sets $S^{\alpha}(\beta),$ there might be a cancellation which allows us to assert that $\psi_A$ belongs to $W^{2,\infty}(B_\frac{1}{2}).$ However, the symmetry condition is slightly restrictive. It requires that there be multiple sectors with the same angle $\alpha$ centered at specific locations $\beta$. For example, if $$A= S^{\frac{\pi}{6}}(0)\cup S^{\frac{\pi}{6}}(\frac{2\pi}{3})\cup S^{\frac{\pi}{6}}(\frac{4\pi}{3}),$$ then we know that $u_{A}\in W^{2,\infty}(B_\frac{1}{2}(0)).$ We now wish to relax the symmetry assumption to allow for: $$A=\cup_{i=1}^m S^{\alpha_i}(\beta_i), $$ with $S^{\alpha_i}(\beta_i)$ non-intersecting. This is the content of the following theorem. 

\begin{thm}
Fix $m\geq 3$.  Assume $$A=\cup_{i=1}^m S^{\alpha_i}(\beta_i), $$ with $S^{\alpha_i}(\beta_i)\cap S^{\alpha_j}(\beta_j)=\phi,$ if $i\not=j.$ 
Then $\psi_A\in W^{2,\infty}(B_\frac{1}{2}(0))$ if and only if $$\sum_{i=1}^m \cot(\alpha_i)\Big[\frac{1}{\csc^2\alpha_i}-2\frac{(\cos\beta_i\cot\alpha_i+\sin(\beta_i))^2}{\csc^4\alpha_i} \Big]=0,$$
 $$\sum_{i=1}^m \cot(\alpha_i)\Big[\frac{1}{\csc^2\alpha_i}-2\frac{(-\sin\beta_i\cot\alpha_i+\cos(\beta_i))^2}{\csc^4\alpha_i} \Big]=0,$$
and
 $$\sum_{i=1}^m \cot(\alpha_i)\Big[\frac{2}{\csc^2\alpha_i}-2\frac{((\cos\beta_i-\sin\beta_i)\cot\alpha_i+\sin(\beta_i)+\cos(\beta_i))^2}{\csc^4\alpha_i} \Big]=0.$$

\end{thm} 

\begin{proof}
The proof reduces to exactly computing $u_{A}$ a ball around $0$ when $A=S^{\alpha}(0)$ for fixed $A$. Using Lemma \ref{Calculation} in the Appendix, we see that if $A=S^\alpha(0), $ $$u_{A}= \cot(\alpha) \Big[\frac{x_1^2+x_2^2}{\csc^2\alpha}-2\frac{(x_2\cot\alpha+x_1)^2}{\csc^4\alpha}\Big]\log|x_1^2+x_2^2|+G$$ with $G\in W^{2,\infty}.$ Let's first see what happens if $A=S^\alpha(0)\cup S^\alpha(\frac{\pi}{2}),$ for some $\alpha$ which is not too big. By linearity, $u_{A}=u_{S^\alpha(0)}+u_{S^\alpha(\frac{\pi}{2})},$ since $S^\alpha(\frac{\pi}{2})$ is just a $90$ degree rotation of $S^\alpha(0).$

In this case we get:$$(x_2\cot\alpha +x_1)^2+(x_1\cot\alpha-x_2)^2=(x_2^2+x_1^2)\csc^2\alpha$$ so that the singularity is completely suppressed. In particular, when $$A=\cup_{i=1}^m S^{\alpha_{i}}(\beta_i)$$ we get a sum of the following form:
$$\sum_{i=1}^m\cot(\alpha_i)\Big[\frac{x_1^2+x_2^2}{\csc^2\alpha_i}-2\frac{\Big((\mathcal{O}_{\beta_i}x)_1\cot\alpha_i+(\mathcal{O}_{\beta_i}x)_2\Big)^2}{\csc^4\alpha_i}\Big].$$
Hence we see that $\psi_A\in W^{2,\infty}$ if and only if this sum vanishes identically. Notice that a quadratic polynomial of two variables is identically 0 if and only if it vanishes on $x_1=0$, $x_2=0,$ and $x_1=x_2.$
Hence we need $$\sum_{i=1}^m \cot(\alpha_i)\Big[\frac{1}{\csc^2\alpha_i}-2\frac{(\cos\beta_i\cot\alpha_i+\sin(\beta_i))^2}{\csc^4\alpha_i} \Big]=0,$$
 $$\sum_{i=1}^m \cot(\alpha_i)\Big[\frac{1}{\csc^2\alpha_i}-2\frac{(-\sin\beta_i\cot\alpha_i+\cos(\beta_i))^2}{\csc^4\alpha_i} \Big]=0,$$
and
 $$\sum_{i=1}^m \cot(\alpha_i)\Big[\frac{2}{\csc^2\alpha_i}-2\frac{((\cos\beta_i-\sin\beta_i)\cot\alpha_i+\sin(\beta_i)+\cos(\beta_i))^2}{\csc^4\alpha_i} \Big]=0.$$

\end{proof}

\subsection{An example}

To illustrate the previous result, we give examples of sets $A$ and $B$ for which $u_{A}\in W^{2,\infty}$ and $u_{B}\not\in W^{2,\infty}$ in a small neighborhood of the origin, though they look very similar!
This can be done by taking $$A=S^{\alpha}(0)\cup S^{\alpha}(\frac{\pi}{2})$$ and $$B=S^{\alpha}(0)\cup S^{\alpha}(\pi)$$ which are pictured below. It is clear that $A$ satisfies the conditions of Theorem 6.2 while $B$ does not. 
\begin{figure}[h]
{\includegraphics[width=5cm, height=4cm]{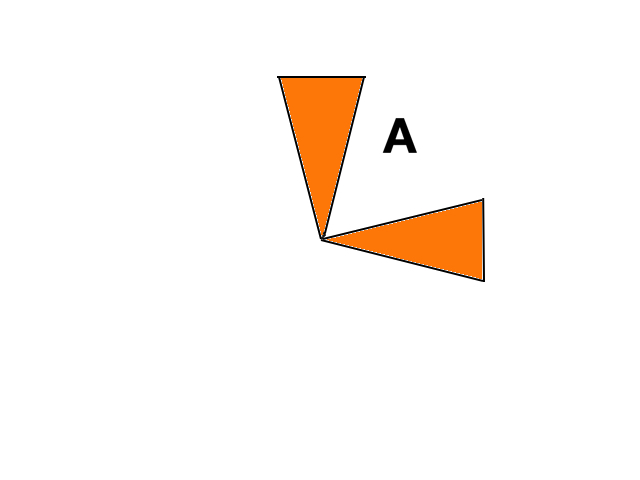}\,\,\,\,\,\,\,\,\,\,\,\,\,\,\,\,\,\,\,\,\,\,\,\,\,\,\,\,\,\,\,\,\,\,\,\,\,\,\,\,\,\,\,\,\,\,\,\,\,\,\,\,\,\,\,\,\,\,\,\,\,\,\,\,\,\,\,\,\,\,\,\,\,\,\,\,\,\includegraphics[width=5cm, height=5cm]{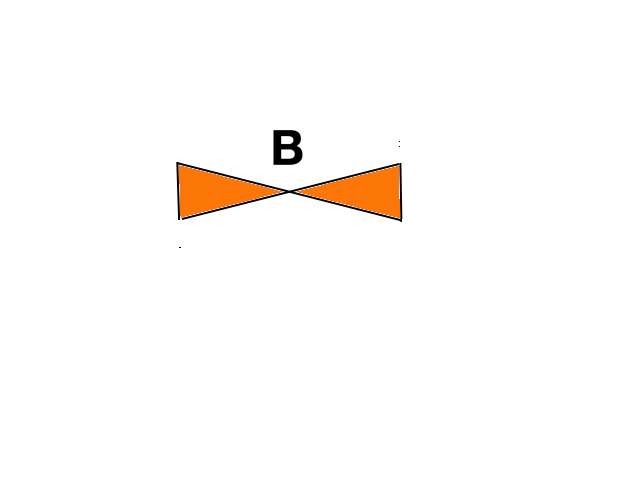}}
\caption{$\psi_{A}\in W^{2,\infty}(B_{\frac{1}{2}}(0))\quad\quad\quad\quad\quad\quad\quad\quad\quad\quad\quad\quad\quad\,\,\,\,\,\,\, \psi_{B}\not\in W^{2,\infty}(B_{\frac{1}{2}}(0))\quad\quad\quad\quad\quad\quad$}
\end{figure}

\section{Acknowledgements}
The author would like to thank Nader Masmoudi, Fang-Hua Lin, and In-Jee Jeong for helpful discussions and the NSF for its generous funding.
\section{Appendix}

In this section we give the computation of the non-smooth part of the solution of $$\Delta \psi=\chi_{A}$$ when $A=S^{\alpha}(0)$ for any $\alpha>0$. 

\begin{lemm}\label{Calculation}
$$\int_0^1 \int_0^{cy_1} \log |x-y|dy_1dy_2=\frac{c}{2}\Big[\frac{x_1^2+x_2^2}{1+c^2} -2\Big(\frac{cx_2+x_1}{1+c^2}\Big)^2\Big]\log|x_1^2+x_2^2| +G$$ with $G\in W^{2,\infty}(B_\frac{1}{2}(0)).$

\end{lemm}

\begin{proof}

A useful fact for this computation is that the anti-derivative of $\log(x^2+A^2)$ with respect to $x$ is: $x\log(x^2+A^2)$ plus a Lipschitz function with Lipschitz constant which is independent of $A$ when $0\leq A\leq 1.$ 

$$\int_0^1 \int_0^{cy_1} \log |x-y|dy_1dy_2=\frac{1}{2}\int_0^1\int_0^{cy_1} \log | (x_1-y_1)^2+(x_2-y_2)^2|dy_1dy_2$$
$$=\frac{1}{2} \int_0^1 (y_2-x_2)\log|(x_1-y_1)^2+(x_2-y_2)^2| \Big|_0^{cy_1}dy_1 + G_1,$$ where $G_1$ is a "good" term belonging to $W^{2,\infty}$. 
Hence, up to a $W^{2,\infty}$ function and multiplication by a constant, the integral in question equals:
$$\int_0^1 (cy_1-x_2) \log | (x_1-y_1)^2+ (x_2-cy_1)^2| dy_1+\int_0^1 x_2 \log|(x_1-y_1)^2+x_2^2|dy_1.$$
Now, notice that $$(x_1-y_1)^2+ (x_2-cy_1)^2=(1+c^2)y_1^2-2y_1(cx_2+x_1)+x_1^2+x_2^2$$
$$=(\sqrt{1+c^2} y_1-\frac{cx_2+x_1}{\sqrt{1+c^2}})^2+x_1^2+x_2^2-\frac{(cx_2+x_1)^2}{1+c^2}$$
$$=(1+c^2)\Big[(y_1-\frac{cx_2+x_1}{1+c^2})^2+ \frac{(x_1^2+x_2^2)(1+c^2)-(cx_2+x_1)^2}{(1+c^2)^2}\Big]$$
In particular, $$x_2\int_0^1 \log |(x_1-y_1)^2+(x_2-cy_1)^2|dy_1$$ $$=x_2\int_0^1 \log\Big[(y_1-\frac{cx_2+x_1}{1+c^2})^2+ \frac{(x_1^2+x_2^2)(1+c^2)-(cx_2+x_1)^2}{(1+c^2)^2}\Big]dy_1+G_2$$ with $G_2\in W^{2,\infty}$. 
Then, $$x_2\int_0^1 \log\Big[(\frac{cx_2+x_1}{1+c^2})^2+ \frac{(x_1^2+x_2^2)(1+c^2)-(cx_2+x_1)^2}{(1+c^2)^2}\Big]dy_1$$
$$=x_2 \log\Big|(\frac{cx_2+x_1}{1+c^2})^2+ \frac{(x_1^2+x_2^2)(1+c^2)-(cx_2+x_1)^2}{(1+c^2)^2}\Big|=x_2\log|x_1^2+x_2^2|+G_3.  $$
So we see that $$\int_0^1 (cy_1-x_2) \log | (x_1-y_1)^2+ (x_2-cy_1)^2| dy_1+\int_0^1 x_2 \log|(x_1-y_1)^2+x_2^2|dy_1$$
$$=c \int_0^1 y_1 \log|(x_1-y_1)^2+(x_2-cy_1)^2|dy_1$$ $$=c\int_0^1y_1\log\Big[(y_1-\frac{cx_2+x_1}{1+c^2})^2+ \frac{(x_1^2+x_2^2)(1+c^2)-(cx_2+x_1)^2}{(1+c^2)^2}\Big]dy_1.$$

Now we note that the antiderivative of $x\log |(x-b)^2+A|$ is $\frac{1}{2}(A^2-b^2+x^2 )\log|(x-b)^2+A|$ plus a Lipschitz function with Lipschitz constant independent of $A$. Hence, computing, the integral we get:
$$\frac{c}{2}\Big[\frac{x_1^2+x_2^2}{1+c^2} -2\Big(\frac{cx_2+x_1}{1+c^2}\Big)^2\Big]\log|x_1^2+x_2^2|.$$ plus $W^{2,\infty}$ parts. 

Now noting that $c=\cot\alpha$ gives us the result. 
\end{proof}


\begin{thebibliography}{10}


\bibitem{BC94}H. Bahouri and J.-Y. Chemin, \'{E}quations de transport relatives \'{a} des champs de veceurs non-lipschitziens et m\'{e}canique des fluides. Arch. Rational. Mech, Anal. 127 (1994), no. 2, 159-181.  

\bibitem{BCD}
H.~Bahouri, J.-Y. Chemin, and R.~Danchin.
\newblock {\em Fourier analysis and nonlinear partial differential equations},
  volume 343 of {\em Grundlehren der Mathematischen Wissenschaften [Fundamental
  Principles of Mathematical Sciences]}.
\newblock Springer, Heidelberg, 2011.

\bibitem{BertozziConstantin}

A. Bertozzi and P. Constantin.
\newblock \emph{Global Regularity for Vortex Patches.}
\newblock Comm. Math. Phys. 1993. 

\bibitem{Chemin}
J.-Y.Chemin.
\newblock{\em Persistency of geometric structures in bidimensional incompressible fluids.}
\newblock Annales scientifiques de l'École normale supérieure. Volume 26. Issue 4. Page 517 - 542


\bibitem{E}

T.M. Elgindi.
\newblock {\em Propagation of Singularities for the 2d incompressible Euler equations.}
\newblock In preparation.

\bibitem{EJ}
T.M. Elgindi and I.J. Jeong.
\newblock {\em On Singular Vortex Patches.}
\newblock In preparation.

\bibitem{GilbargTrudinger}
D. Gilbard and N.S. Trudinger.
\newblock {\em Elliptic Partial Differential Equations of Second Order.}
\newblock Grundlehren, Vol. 224, Springer-Verlag, Berlin, 1983.

\bibitem{HanLin}
Q. Han and F.H. Lin.
\newblock{\em Elliptic Partial Differential Equations.}
\newblock{Courant Lecture Notes,} Vol. 1.

\bibitem{KS}
A. Kiselev and V. Sverak.
\newblock{\em Small scale creation for solutions of the incompressible two dimensional Euler equation}
\newblock {Annals of Math}, 2015.

\bibitem{Xu}
X. Xu.
\newblock{\em Fast growth of the vorticity gradient in symmetric smooth domains for 2D incompressible ideal flow}
\newblock{arXiv preprint}, 2014.

\bibitem{Yoneda}
T. Itoh, H. Miura, and T. Yoneda.
\newblock{The growth of the vorticity gradient for the two-dimensional Euler flows on domains with corners}
\newblock{arXiv preprint,} 2016.

\bibitem{Yudovich63}
V.~I. Yudovich.
\newblock {\em Non-stationary flows of an ideal incompressible fluid}.
\newblock {\em u Z. Vy cisl. Mat. i Mat. Fiz.}, 3:1032--1066, 1963.



\end{thebibliography}
\end{document}